\documentclass{amsart}

\usepackage{mathscinet}
\usepackage{amsmath,amssymb}
\usepackage[utf8]{inputenc}
\newtheorem{theorem}{Theorem}[section]
\usepackage{color}

\usepackage{comment}
\usepackage{enumitem}
\usepackage[hidelinks]{hyperref}

\newtheorem{lemma}[theorem]{Lemma}

\theoremstyle{definition}
\newcount\skewfactor
\def\mathunderaccent#1#2 {\let\theaccent#1\skewfactor#2
\mathpalette\putaccentunder}
\def\putaccentunder#1#2{\oalign{$#1#2$\crcr\hidewidth
\vbox to.2ex{\hbox{$#1\skew\skewfactor\theaccent{}$}\vss}\hidewidth}}

\newtheorem{definition}[theorem]{Definition}

\newtheorem*{assumption*}{Assumption}

\newtheorem{fact}[theorem]{Fact}

\theoremstyle{remark}

\newtheorem{remark}[theorem]{Remark}

\newcommand{\AP}{{\mathrm{AP}}}
\newcommand{\IS}{{\mathrm{IS}}}
\newcommand{\APp}{{\mathrm{AP}_\lambda}}
\newcommand{\veps}{\varepsilon}

\DeclareMathOperator{\Sym}{Sym}

\DeclareMathOperator{\dom}{dom}
\newcommand{\bfa}{{\mathbf{a}}}
\newcommand{\bfb}{{\mathbf{b}}}
\newcommand{\bfc}{{\mathbf{c}}}

\newcommand{\bfC}{{\mathbf{C}}}

\usepackage{bm}

\newcommand{\btpi}{{\bm{\tilde\pi}}}

\newcommand{\btphi}{{\bm{\tilde\phi}}}
\newcommand{\bfpi}{{\bm{\pi}}}
\newcommand{\bfB}{{\mathbf{B}}}
\newcommand{\bftB}{{\tilde{\mathbf{B}}}}

\title[Nowhere trivial automorphisms of ${P(\lambda)/[\lambda]^{<\lambda}}$]{Nowhere trivial automorphisms of $P(\lambda)/[\lambda]^{<\lambda}$,\\ for $\lambda$ inaccessible.}
\author{Jakob Kellner}
\address{Technische Universität Wien (TU Wien).}
\email{jakob.kellner@tuwien.ac.at}
\urladdr{\url{http://dmg.tuwien.ac.at/kellner/}}
\author{Saharon Shelah}
\address{The Hebrew University of Jerusalem and Rutgers University.}
\email{shlhetal@mat.huji.ac.il}
\urladdr{\url{http://shelah.logic.at/}}

\thanks{The first author was funded by FWF Austrian science funds projects P33420	and P33895. The second author was partially support by the Israel Science Foundation (ISF) grant no: 2320/23. This is publication Sh:1251 in the second author's list.}

\date{2024-03-04}

\begin{document}
\maketitle

\newcommand{\myte}{T1}
\newcommand{\mytz}{T2}
\newcommand{\thetext}{If $\lambda$ is (strongly) inaccessible and $2^\lambda = \lambda^+$,
then there is a nowhere trivial automorphism of
the Boolean algebra 
$\mathcal P(\lambda)/[\lambda]^{<\lambda}$.}

\section{Introduction}

We investigate the rigidity of the Boolean algebra
$P(\lambda)/[\lambda]^{<\lambda}$,
for $\lambda$ inaccessible.

For $\lambda=\omega$ there is extensive
literature on this topic (see, e.g., the survey~\cite{farah2024corona});
some general results on $P(\lambda)/[\lambda]^{<\kappa}$
can be found in~\cite{MR3480121}. In~\cite{Sh:1224} it was shown, for
$\lambda$ inaccessible and 
$2^\lambda=\lambda^{++}$, that conistently
every automorphism is densely trivial.

In this paper we show:
\begin{equation}
\tag{Thm.~\ref{thm:a}}\parbox{0.8\textwidth}{\thetext}    
\end{equation}

Note that the weaker variant  ``there is a \emph{nontrivial} automorphism''
follows from~\cite[Lem.~3.2]{Sh:990} (the proof there was faulty, and fixed in~\cite{Sh:990a}); and for $\lambda$ measurable, a proof (again only for ``nontrivial'') was given in~\cite{Sh:1224}.

We also show:
\begin{equation}
\tag{Thm.~\ref{thm:b}}\parbox{0.8\textwidth}{It is consistent that $\lambda$ is inaccessible, $2^\lambda$ an arbitrary regular cardinal, and that there is a nowhere trivial automorphism of $\mathcal P(\lambda)/[\lambda]^{<\lambda}$.}    
\end{equation}

\section{Notation} 
We will always assume that
%\begin{assumption*}
$\lambda$ is inaccessible.    
%\end{assumption*}

For $A\subseteq \lambda$, 
$[A]^{<\lambda}$ denotes the subsets 
of $A$ of size less than $\lambda$;
and $[A]^{\lambda}$ those of size $\lambda$.
With $[A]$ we denote the equivalence class of 
$A$ modulo $[\lambda]^{<\lambda}$.
We write $A=^*B$ for $[A]=[B]$, and $A\subseteq^* B$ for
$|A\setminus B|<\lambda$.

However, we also use $f[A]:=\{f(a):\, a\in A\}$.
So for example $[f[A]]$ is the equivalence class of the $f$-image of $A$.
$f\in \Sym(X)$ means that $f:X\to X$ is bijective.

We consider  $\mathcal P(\lambda)/[\lambda]^{<\lambda}$ as Boolean algebra.
A (Boolean algebra)  automorphism $\pi$ of $\mathcal P(\lambda)/[\lambda]^{<\lambda}$
is called trivial on $A$ (for $A\in [\lambda]^\lambda$)
if there is an $f\in\Sym(\lambda)$ such that
$\pi([B])=[f[B]]$ for all $B\subseteq A$. 
%(And $\pi$ is called trivial, if it is trivial on $\lambda$.)
$\pi$ is called nowhere trivial, if there is no such pair $(f,A)$.

For $\delta\le\lambda$, $C\subseteq \delta$ closed and nonempty, and $\alpha\in C$, we set 
\[I^*(C\subseteq  \delta, \alpha):=\big\{\beta:\,  \alpha\le \beta<\min\big((C\cup\{\delta\})\setminus (\alpha+1)\big)\big\}.\]
So the $I^*(C\subseteq  \delta, \alpha)$, for $\alpha\in C$,
form an increasing interval partition of $\delta\setminus\min(C)$.

\section{Approximations}

In this section, we define the set $\AP$ of ``approximations''. 
An approximation $\bfa$ will induce a 
``partial monomorphism'' $\btpi^\bfa$ defined on some $\bftB^\bfa$ which is 
trivial, i.e., generated by some $\bfpi^\bfa\in\Sym(\lambda)$.
We will use such approximations to build a  nowhere trivial automorphism $\btphi$ as limit (i.e., $\btphi\restriction \bftB^ \bfa=\btpi^ \bfa$), cf.~Fact~\ref{fact:central}.

% For such a $C$, we will 
% pick an 
% $\cB'_\veps\subseteq \mathcal P(I_\veps)$
% and let $J'_\veps\subseteq \cB'_\veps$
% contain all bounded (in $I_\veps$) 
% elements of $\cB'_\veps$.
% Consider the BA-automorphisms of
% $\cB'_\veps$ that leave $J'_\veps$
% invariant \wtf{that means: iff, 
% or implies?}, and pick a group
% $G'_\veps$ of such automorphisms.
% We will further require that each
% $g'\in G'_\veps$ 
% is induced by a bijection 
% $j'(g')$
% of $I_\veps$: For $x\in \cB_\veps$,
% $g'(x)=j'(g')''x$.

\begin{definition}\label{def:AP}
    $\AP$ is the set of objects $\bfa$ consisting of
    $\bf\bfC^\bfa$, $\bfpi^\bfa$
    and ${\bfB}^\bfa$, % ${\bfG}^\bfa$,
    such that: 
    \begin{itemize}
        \item $\bfpi^\bfa\in\Sym(\lambda)$.
        \item $\bf\bfC^\bfa\subseteq \lambda $ club 
        such  that $\bfpi^\bfa\restriction 
        \veps\in\Sym(\veps)$
        for all $\veps\in \bf\bfC^\bfa$.
                \item ${\bfB}^\bfa$ is a subset of $\mathcal{P}(\lambda)$.
    \end{itemize}
\end{definition}
$\bfpi^\bfa$ induces a (trivial) automorphism of
$P(\lambda)/[\lambda]^{<\lambda}$, and 
$\btpi^\bfa$ is the 
restriction of this automorphism to $\bfB^\bfa$:
\begin{definition}
    \begin{itemize}
        
        \item 
$\bftB^\bfa:=\bfB^\bfa/[\lambda]^{<\lambda}=\{[A]:\, A\in \bfB^\bfa\}$.
\item 
$\btpi^\bfa: \bftB^\bfa\to P(\lambda)/[\lambda]^{<\lambda}$ is defined by 
$[A]\mapsto [\bfpi^\bfa[A]]$.
%So the $I^C_\alpha$ form an increasing interval partition of $\delta\setminus\min(C)$.
\item For $\bfa\in\AP$ and $\veps\in \bf\bfC^\bfa$
we set $I^\bfa_\veps:=I^*(\bf\bfC^\bfa\subseteq\lambda,\veps)$.
% \msucc_{\bf\bfC^\bfa}(\veps):=\min(C\setminus (\veps+1))
% \quad
% \text{and}
% \quad
%$I^\bfa_\veps=\{\alpha\ge \veps:\, 
%\alpha<\min(C\setminus (\veps+1))
%\}$.
    \end{itemize}
\end{definition}
So the $I^\bfa_\veps$ form an increasing interval partition of $\lambda\setminus\min(\bf\bfC^\bfa)$;
and $\bfpi^\bfa\restriction I^\bfa_\veps\in\Sym(I^\bfa_\veps)$.
  
% OLDCRAP START:

% \begin{definition}\label{def:le}
%     $\bfb\ge_\AP \bfa$, if $\bfa,\bfb\in \AP$ and
%     there is an $i^*(\bfa,\bfb)<\lambda$ and, 
%     for $A\in \bfB^\bfa$ an $i^*_2(A,\bfa,\bfb)<\lambda$ such that
%     \begin{enumerate}
%         \item\label{item:le1}  
%         $\bf\bfC^\bfb\setminus i^*(\bfa,\bfb)\subseteq \bf\bfC^\bfa$.
%             \item\label{item:le2} 
            
%         $\bfpi^\bfb\restriction I^\bfa_\veps
%         =\bfpi^\bfa\restriction I^\bfa_\veps$ 
%             for  all $\veps\in \bf\bfC^\bfb\setminus i^*(\bfa,\bfb)$.
%         \item $\bfB^\bfb\supseteq \bfB^\bfa$; 
%         %$\bfG^\bfb\supseteq \bfG^\bfa$; 
%         \item\label{item:le0} 
%         $\btpi^\bfb$ extends $\btpi^\bfa$.
%         More specifically:
        
%         If $A\in \bfB^\bfa$, then
%         $\bfpi^\bfa[A]=\bfpi^\bfb[A]$ above $i^*_2(A,\bfa,\bfb)$.
%     \end{enumerate}
% \end{definition}
% (Of course we could just say $\bf\bfC^\bfb \subseteq^* \bf\bfC^\bfa$ and 
% ``for sufficiently large $\veps$'' etc, instead of explicitly
% using $i^*$ and $i^*_2$, but this more verbose definition allows us to refer to the bound explicitly later on.)

% OLDCRAP END 
% OLDCRAP START:

\begin{definition}\label{def:lenew}
    $\bfb\ge_\AP \bfa$, if $\bfa,\bfb\in \AP$ and
    %for $A\in \bfB^\bfa$ an $i^*_2(A,\bfa,\bfb)<\lambda$ such 
    \begin{enumerate}
        \item  
        $\bf\bfC^\bfb\subseteq^* \bf\bfC^\bfa$.
            \item     
        $\bfpi^\bfb\restriction I^\bfa_\veps
        =\bfpi^\bfa\restriction I^\bfa_\veps$ 
            for  all but boundedly many $\veps\in \bfC^\bfb$.
        \item $\bfB^\bfb\supseteq \bfB^\bfa$, and 
        %$\bfG^\bfb\supseteq \bfG^\bfa$; 
        %\item 
        $\btpi^\bfb$ extends $\btpi^\bfa$.
            \end{enumerate}
\end{definition}
I.e., if $A\in \bfB^\bfa$, then
        $\bfpi^\bfa[A]=^*\bfpi^\bfb[A]$.

% (Note that generally 
% $\check\pi^\bfb_\veps(\zeta)\ne 
% \check\pi^\bfa_\veps(\zeta)$ for $\lambda$ many $\zeta$.)
$\le_\AP$ is a nonempty quasi order.

\begin{lemma}\label{lem:trivial2}
    If $(\bfa_i)_{i<\delta}$ is an
    $\le_\AP
    $ increasing chain 
    such that $\bigcup_{i<\delta} \bftB^{\bfa_i}=P(\lambda)/[\lambda]^{<\lambda}$, then
    $\btphi:=\bigcup_{i<\delta} \btpi^{\bfa_i}$ is
    an Boolean algebra monomorphism of $P(\lambda)/[\lambda]^{<\lambda}$.

    If additionally $\bigcup_{i<\delta} \btpi^{\bfa_i}[\bftB^{\bfa_i}]=P(\lambda)/[\lambda]^{<\lambda}$, then $\btphi$ is an automorphism.
\end{lemma}
\begin{proof}
We use $\vee$ and $^c$ for the Boolean-algebra-operations,
i.e., $[A\cup B]=[A]\vee[B]$, 
%  $[A\cap B]=[A]\wedge [B]$
and $[A]^c=[\lambda\setminus A]$.
It is enough to show that $\btphi$ is injective,
honors $\vee$ and $^c$, and maps $[\emptyset]$ to itself.

    For $X_1,X_2$ in $P(\lambda)/[\lambda]^{<\lambda}$
    there is an 
    $i<\delta$ and some
    $A_1,A_2,A_\text{union}$ in $\bfB^{\bfa_i}$,
    such that $[A_j]=X_j$ for $j=1,2$ and $[A_\text{union}]=[A_1\cup A_2]=X_1\vee X_2$.
    Then 
    \[
    \bfpi^{\bfa_i}[A_\text{union}]=^*
    \bfpi^{\bfa_i}[A_1\cup A_2] =   
    \bfpi^{\bfa_i}[A_1]\cup  \bfpi^{\bfa_i}[A_2],
    \]
    and
    \begin{gather*}
    \btphi(X_1\vee X_2)=\btpi^{\bfa_i}
    ([A_\text{union}])=
    [\bfpi^{\bfa_i}[A_\text{union}]]=\\
    =\btpi^{\bfa_i}([A_1])\vee \btpi^{\bfa_i}([A_2])=
    \btphi(X_1)\vee \btphi(X_2).
    \end{gather*}
    If $X_1\ne X_2$, i.e., $A_1\ne^* A_2$, 
    then $\bfpi^{\bfa_i}[A_1]\ne^*  \bfpi^{\bfa_i}[A_2]$, i.e., $\btphi(X_1)\ne\btphi(X_2)$.

    Similarly we can show $\btphi([\lambda\setminus A_1])=
    \btphi([A_1])^c$ and $\btphi([\emptyset])=[\emptyset]$.    
\end{proof}

\begin{definition}
    For a pair $(f,A)$ with $A\in[\lambda]^\lambda$ and
    $f\in\Sym(\lambda)$, we say $\bfa\in \AP$ ``spoils $(f,A)$'',
    if there is an $A'\in[A]^\lambda\cap \bfB^\bfa$
    such that $|\bfpi^\bfa[A']\cap f[A']|<\lambda$.
\end{definition}
If  $\btphi$  is an automorphism extending 
such a $\btpi^\bfa$, then $f$ cannot witness that $\btphi$ is trivial on $A$.
Therefore:
\begin{fact}\label{fact:central}
   If $(\bfa_i)_{i<\delta}$ is an
    $\le_\AP
    $ increasing chain 
    such that 
    \begin{itemize}
        \item $\bigcup_{i<\delta} \bftB^{\bfa_i}=
    \bigcup_{i<\delta} \btpi^{\bfa_i}[\bftB^{\bfa_i}]=
    P(\lambda)/[\lambda]^{<\lambda}$,  and
    \item for every $(f,A)$ there is an $i<\delta$ such that
    $\bfa_i$ spoils $(f,A)$,
    \end{itemize}
    then
    $\btphi:=\bigcup_{i<\delta} \btpi^{\bfa_i}$ is
    a nowhere trivial Boolean algebra automorphism of $P(\lambda)/[\lambda]^{<\lambda}$.
\end{fact}
We will use this fact both in the case $2^\lambda=\lambda^+$, as well as in the forcing construction to get a nowhere trivial automorphism.

\medskip
We will often modify an $\bfa\in \AP$
by replacing $\bfB^\bfa$ with another $B\subseteq  P(\lambda)$. Let the result be $\bfb$.
We call $\bfb$ ``$\bfa$ with
$\bfB$ replaced by $B$'', or
``$\bfa$ with $X$ added to
$\bfB$'' in case $B=\bfB^\bfa\cup\{X\}$. Obviously
$\bfb\in \AP$, and if $B\supseteq \bfB^\bfa$ then 
$\bfb\ge _\AP \bfa$.

Similarly we can get a stronger approximation by thinning out $\bfC$. To summarize:
\begin{fact}\label{fact24}
    If $\bfa\in\AP$,
        $D\subseteq \bf\bfC^\bfa$ club, and 
        $B\subseteq P(\lambda)$ with $B\supseteq \bfB^\bfa$.
    Then  $\bfb\ge_\AP \bfa$, for the $\bfb$
    defined by 
    $\bfpi^\bfb:=\bfpi^\bfa$,
    $\bfC^\bfb:= D$ 
    and $\bfB^\bfb:=B$.
\end{fact}

% \begin{lemma}\label{lem:trivial}
%     Assume 
%          $\bfa\in\AP$,
%         $D\subseteq \bf\bfC^\bfa$ club, and 
%          $X,Y\subseteq \lambda$.
%     Then there is a $\bfb\le_\AP \bfa$ 
%     such that 
%     $\bf\bfC^\bfb= D$,
%     $X\in \bfB^\bfb$, and 
%     $Y=\bfpi^\bfb[Z]$ for some $Z\in\bfB^\bfb$.
% \end{lemma}
% \begin{proof}
%     We set 
%     $\bf\bfC^\bfb:= D$, $\bfB^\bfb:=\bfB^\bfa\cup\{X,(\bfpi^\bfa)^{-1}Y\}$, and
%     $\bfpi^\bfb:=\bfpi^\bfa$. As we do not change $\bfpi$, all the requirements
%     for $\bfb>\bfa$ are trivially met.
% \end{proof}

In the definition of $\le_\AP$ we require that 
some things hold ``apart from a bounded set'', or equivalently,
``above some $\alpha$''.
We say that $\alpha$ is good for an increasing sequence of $\bfa_i$, 
if the requirements for each pair are met above $\alpha$.
We will generally only be able to find such an $\alpha$ for ``short sequences'':
%More exactly:
\begin{definition}\label{def:good}
\begin{enumerate}
\item%\label{item:APp}
$\APp$ is the set of $\bfa\in\AP$ such that $|\bfB^\bfa|\le\lambda$.  Analogously for $\AP_{<\lambda}$.
    \item $(\bfa_i)_{i\in J}$ is a ``short sequence'', if
    $J<\lambda$ (or more generally, $J$ is a set of ordinals with $|J|<\lambda$), each $\bfa_i\in \AP_{<\lambda}$,
    and the sequence is $\le_\AP$-increasing, i.e., $j>i$ in $J$
    implies $\bfa_j\ge_\AP \bfa_i$.
    \item 
%Let $J$ an ordered set, and 
Let $\bar\bfa:=(\bfa_i)_{i\in J}$ be
short. %an  $<_\AP$-increasing sequence.
We say that $\alpha$ is good for $\bar\bfa$, if
for all $i\le k$ in $J$:
\begin{enumerate}
    \item $\alpha\in \bfC^{\bfa_i}$.
    \item\label{item:epitrjqw} $\bfC^{\bfa_k}\subseteq \bfC^{\bfa_i}$ above $\alpha$.
    (I.e., $\beta\ge\alpha$ and $\beta\in \bfC^{\bfa_k}$
    implies $\beta\in\bfC^{\bfa_i}$.)
    \item $\bfpi^{\bfa_k}\restriction I^{\bfa_i}_\veps
    =\bfpi^{\bfa_i}\restriction I^{\bfa_i}_\veps$
    for all $\veps\ge\alpha$ in $\bfC^{\bfa_k}$.
    \item \label{item:epitrjqwz}
        $\bfpi^{\bfa_i}[A]\setminus\alpha=\bfpi^{\bfa_k}[A]\setminus\alpha$,
        for all $A\in \bfB^{\bfa_i}$.
\end{enumerate}  
\item 
For $\bfa,\bfb$ in $\AP_{<\lambda}$, we say 
$\bfb>_\zeta \bfa$, if $\zeta$ is good for the 
sequence $\langle\bfa,\bfb\rangle$.
%(Here, $\bfB^\bfb$ is irrelevant, we can replace it with
%$\bfB^\bfa$ to make sure that the sequence is short.)
\end{enumerate}
\end{definition}

So in particular if $\bfb$ is the result of 
enlarging $\bfB$ in $\bfa$, then $\bfb>_\zeta \bfa$
for all $\zeta\in \bfC^\bfa$.

\begin{fact}\label{fact:goodisclub}
%Let $\bfa,\bfb$ be in $\AP$.
\begin{enumerate}
    \item %$\bfb>_\zeta \bfa$ implies $\bfb\ge_\AP\bfa$.    
    If $\bfa\in \AP_{<\lambda}$, then
    $\bfb\ge_\AP \bfa$ iff  $(\exists \zeta\in\lambda)\,\bfb>_\zeta \bfa$.
\item\label{item:ouhjqw} 
If $\bar\bfa=(\bfa_i)_{i\in J}$ is short, 
then $\{\alpha\in\lambda:\, \alpha\ \mathrm{ good\ for } \ \bar\bfa\}$ is club, more concretely it is $\bigcap_{i<\delta} \bfC^{\bfa_i}\setminus \alpha^*$ for some $\alpha^*<\lambda$.
\end{enumerate}
\end{fact}
% \begin{proof}
%     We just show (\ref{item:ouhjqw}). For each pair $i<j$ in $J$
%     and each $A\in\bfB^{\bfa_i}$ there is a 
%     minimal bound
%     $\alpha^*(i,j,A)$ satisfying 
%     (\ref{item:epitrjqw}--\ref{item:epitrjqwz})
%     of Definition~\ref{def:good}.
%     Let $\alpha^*$ be the supremum of these $\alpha^*(i,j,A)$.
%     Then $\alpha$ is good iff 
%     $\alpha\in \bigcap_{i<\delta} \bfC^{\bfa_i}\setminus \alpha^*$.
% \end{proof}

\begin{lemma}\label{lem:shortlimit}
If $\bar\bfa$ is short, then is has an $\le_\AP$-upper-bound 
$\bfb\in\AP_{<\lambda}$.
\end{lemma}

\begin{proof}
Set $D:=\bigcap_{i\in J} C^{\bfa_i}$, and $\zeta_0$ be the smallest
$\bar\bfa$-good ordinal. So in particular $\zeta_0\in D$; and any $\zeta\ge \zeta_0$ is in $D$ iff it is $\bar\bfa$-good.

Fix for now some $\zeta\in D\setminus \zeta_0$. Let $\zeta^+$ be the $D$-successor of $\zeta$. 

For $i\in J$, set $\gamma(\zeta,i)$ to be the $\zeta$-successor of $C^{\bfa_i}$.
Then the sequence $\gamma(\zeta,i)$  is weakly increasing with $i\in J$ and has limit $\zeta^+$.
If $\alpha<\gamma(\zeta,i)$ (we also say ``$\alpha$ is stable at $i$''), then 
$\bfpi^{\bfa_i}(\alpha)=\bfpi^{\bfa_j}(\alpha)$
for all $j>i$ in $J$. 

% For every $\alpha$ with $\zeta\le \alpha<\zeta^+$ and for $i\in J$ there is a unique
% $\beta(\alpha,i)\in C^{\bfa_i}$ with $\alpha\in I^{\bfa_i}_{\beta(\alpha,i)}$.
% The sequence $\beta(\alpha,i)$ is weakly decreasing with $i\in J$ and eventually constantly $\zeta$. More concretely, $\beta(\alpha,i)=\zeta$ iff $\gamma(\zeta,i)>\alpha$. 

We define $\pi^{\mathrm{lim}}(\alpha)$ for all $\alpha\ge \zeta_0$ as 
$\bfpi^{\bfa_i}(\alpha)$ for some $i$ stable for $\alpha$.

\medskip

To summarize: Whenever $I:=\zeta^+\setminus \zeta$ for some $\zeta\in D\setminus\zeta_0$ with $\zeta^+$ the $D$-successor, we get:
\begin{enumerate}
\item\label{item:oinkoink} $(\forall \alpha\in I)\, (\exists i\in J)\, (\forall j>i)\, \pi^{\mathrm{lim}}(\alpha)= \bfpi^{\bfa_i}(\alpha)$. 

    \item\label{item:oueqitjhu3q} $\pi^{\mathrm{lim}}\restriction I\in\Sym(I)$.
    \item\label{item:oueqitjhu3q2} If $i\in J$ and $A\in \bfB^{\bfa_i}$, then
    $\pi^{\mathrm{lim}}[A']=\bfpi^{\bfa_i}[A']$ where $A':=A\cap I$. 
\end{enumerate}

%We have already seen~(\ref{item:oinkoink}). 
For (\ref{item:oueqitjhu3q}), note that $\bfpi^{\bfa_i}\in\Sym(I)$ for all $i\in J$.
If $\alpha_1\ne \alpha_2\in I$, then there is an $i$ in $J$ stable
for both, and $\bfpi^{\bfa_i}(\alpha_1)
\ne \bfpi^{\bfa_i}(\alpha_2)$. So
$\bfpi^{\mathrm{lim}}$ is injective.
And if $\alpha_1\in I$ and $i$ in $J$ stable for $\alpha_1$, then there
is an $\alpha_2\in I^{\bfa_i}_\zeta$ with $\bfpi^{\mathrm{lim}}(\alpha_2)=\bfpi^{\bfa_i}(\alpha_2)=\alpha_1$,
so $\pi^{\mathrm{lim}}$ is surjective.

For (\ref{item:oueqitjhu3q2}):
    %We have to show $B:=\bfpi^{\bfa_i}[I\cap A]=\pi^{\mathrm{lim}}[I\cap A]$.
    %As 
    %Fix any $i\in J$ and set 
    Set $B:=\bfpi^{\bfa_{i}}[A']$.
    As $I$ is above the good $\zeta_0$, we have:
    $B=\bfpi^{\bfa_j}[A']$ for all $j\in J$
    with $j>i$.
    So for $\alpha\in A'$, 
    all $\bfpi^{\bfa_j}(\alpha)$ are in $B$, and 
    also stabilize to $\pi^{\mathrm{lim}}(\alpha)$, which therefore has to be in $B$.
    Analogously, we get: If $\alpha\in I\setminus A$, then 
    $\bfpi^{\bfa_j}(\alpha)\ne B$ stabilizes to 
    $\pi^{\mathrm{lim}}(\alpha)$, which therefore is not in $B$.
    As $\pi^{\mathrm{lim}}[I]=I$, we get $\pi^{\mathrm{lim}}[I\cap A]=B$.

\medskip

We can now define $\bfb$ as:
\[
\bfC^\bfb:=D\setminus\zeta_0
;\quad
\bfpi^\bfb(\alpha)=\begin{cases}
    \alpha&\text{if }\alpha<\zeta_0\\
\pi^\mathrm{lim}(\alpha)&\text{otherwise;}\end{cases}
\quad
\bfB^\bfb:=\bigcup_{i\in J}\bfB^{\bfa_i}.
\qedhere
\]
%
%Now pick any $\delta^c\in D\cap E\setminus (\delta^b+1)$, and
%set $\pi^c:=\pi^b\cup %\pi^{\mathrm{lim}}\restriction (\delta^c\setminus \delta^b)$.
\end{proof}

\section{Initial segments}

We will work with initial segments of approximations (without the $\bfB$ part):
\begin{definition}
\begin{itemize}
    \item 
    An ``initial segment'' 
    $b$ consists of
    a ``height'' $\delta^b$, a closed $C^b\subseteq \delta^b$ (possibly empty), and a $\pi^b\in\Sym(\delta^b)$ such
    that $\pi^b\restriction \zeta\in\Sym(\zeta)$
    for all $\zeta\in C^b$.
\item 
    The set of initial segments is called $\IS$.
\item
    $b>_\IS a$,  if  
    $\delta^b>\delta^a$, $\delta^a\in C^b$, 
    $C^b\cap\delta^a=C^a$, and 
    $\pi^b\restriction\delta^a=\pi^a$.
    \item $b\ge_\IS a$ if $b>_\IS a$ or $b=a$.

    \item
   For $\zeta\in C^b$, we set
$I^b_\zeta:=I^*(C^b\subseteq \delta^b, \zeta)$. 
    \end{itemize}
\end{definition}
So the $I^b_\zeta$ form an increasing interval partition of $\delta^b\setminus \min(C^b)$, and $\pi^b\restriction
I^b_\zeta\in\Sym(I^b_\zeta)$.

$\le_\IS$ is a partial order.

\medskip

Some trivialities:
\begin{fact}\label{fact:istriv1}
    Assume that $\bar b=(b_i)_{i<\xi}$,
    with $\xi\le \lambda$ limit, is an $<_\IS$-increasing sequence.
    \begin{enumerate}
        \item\label{item:ounewgew} 
    If $\xi<\lambda$, then the following
    $b_\xi\in  \IS$ is the    $\le_\IS$-supremum of $\bar b$, and we call it ``the limit'' of $\bar b$:
    $\delta^{b_\xi}:=
    \bigcup_{i<\xi} \delta^{b_i}$,
    $C^{b_\xi}:=
    \bigcup_{i<\xi} C^{b_i}$ and
    $\pi^{b_\xi}:= 
    \bigcup_{i<\xi} \pi^{b_i}$.

        \item \label{item:ounewgewb}
    If $\xi=\lambda$, then
    to each $B\subseteq P(\lambda)$ there is a 
    $\bfb\in\AP$ as follows, which we call
    ``a limit'' of $\bar b$: 
    $\bfC^\bfb:=
    \bigcup_{i<\lambda} C^{b_i}$
    $\bfpi^\bfb:= 
    \bigcup_{i<\lambda} \pi^{b_i}$
    and  $\bfB^\bfb:=B$.
      \end{enumerate}
\end{fact}

Let us call an $<_\IS$-increasing sequence $\bar b$
``continuous'' if 
$b_\gamma$ is the limit of $(b_\alpha)_{\alpha<\gamma}$
for all limits $\gamma<\delta$.  We will only use continuous sequences.

\begin{definition}\label{def:strongext}
    Let $\bfa\in \AP_{<\lambda}$
    and
    $b\in \IS$ with $\delta^b\in \bfC^\bfa$. 
    We say $c>_{\bfa} b$,
    if the following holds:
    \begin{itemize}
    \item 
    $c>_\IS b$.
    \item
    $(C^{c}\cup \{\delta^{c}\}) \setminus 
    \delta^{b} \subseteq\bfC^\bfa$.
    \item For all 
    $\zeta\in C^{c}\setminus 
    \delta^{b}$,
    $\pi^{c}\restriction I^\bfa_\zeta
    = \bfpi^\bfa \restriction I^\bfa_\zeta$.
    \item For all $A\in \bfB^\bfa$,
    $\pi^{c}[A'] 
    =\bfpi^{\bfa}[A']$
    where we set 
    $A':=A\cap \delta^{c}\setminus \delta^{b}$.
    \end{itemize}
For a short $\bar\bfa$ (with index set $J$) 
we say $c>_{\bar\bfa} b$ if $c>_{\bfa_i} b$ for all $i\in J$.
\end{definition}

%The following is straightforward to check:

\begin{lemma}\label{lemma:tr335}
%\begin{facts}\label{facts:triv33}
Let $\bfa,\bfb$ in $\AP_{<\lambda}$ and $c$, $d_i$ ($i<\lambda$) in $\IS$.
\begin{enumerate}
\item\label{item:transe}
%If $d_1<_\bfa d_2$ and $d_2<_\bfa d_3$, then $d_1<_\bfa d_3$.
$<_\bfa$ is a partial order
\item\label{item:tr334} If $\zeta<\lambda$ and 
    $(d_i)_{i\in\zeta}$ is a  $>_\IS$-increasing sequence
    such that $d_i>_{\bfa}c$ for all $i<\zeta$, then also the limit $d_\zeta$ satisfies
    $d_\zeta>_{\bfa}c$.
\item\label{item:tr331} %If $\bfb$ is the same as $\bfa$ only with enlarged $\bfB$, 
If 
$\bfb >_{\delta^c} \bfa$,
then $d>_{\bfb} c$ implies $d>_{\bfa} c$.
%    \item If $b$ is good for $\bfa$ and $c>_\bfa b$, then $c$ is good for $\bfa$. 
\item\label{item:oldlem}    Assume $\bar c:=(c_i)_{i\in\lambda}$
    is a continuous increasing sequence in $\IS$ such that for some $i_0<\lambda$ we have
    $c_i<_\bfa c_{i+1}$ for all $i>i_0$. 
     
    Then 
    any limit $\bfc\in \AP$ of the $\bar c$ with $\bfB^{\bfc} \supseteq \bfB^{\bfa}$ satisfies 
    $\bfc >_\AP \bfa$. 
    \item\label{item:ISextex}
        Let $\bar \bfa$ be short,   
    $b\in \IS$, $\delta^b$ good for $\bar\bfa$
    and $E\subseteq \lambda$ club.
   
    Then
there is a $c>_{\bar \bfa} b$ with $\delta^c\in E$
and $C^c=C^b\cup\{\delta^b\}$
\end{enumerate}
\end{lemma}

\begin{proof}
      For~(\ref{item:ISextex}), use (the proof of) Lemma~\ref{lem:shortlimit}: 
      Pick any $\delta^c\in D\cap E\setminus (\delta^b+1)$ and set
      $C^c=C^b\cup\{\delta^b\}$ and
      $\pi^c=\pi^\mathrm{lim}\restriction\delta^c$.

      The rest is straightforward. 
%set $\pi^c:=\pi^b\cup %\pi^{\mathrm{lim}}\restriction (\delta^c\setminus \delta^b)$.
    %   For~(\ref{item:oldlem})
    %   show by induction
    % on $j\ge i+1$ that  $c_i<_\bfa c_j$ for all $i<j<\lambda$.
    % For successors, that follows from~(\ref{item:transe}),
    % for limits from~(\ref{item:tr334}).
\end{proof}

% We will use item~(\ref{item:oldlem}) in the following form:
% \begin{lemma}
% Assume $\bar \bfa^i:=(\bfa^i_j)_{j\in J^i}$ are short 
%     sequences for $i<\lambda$, such that for 
%     $i<k<\lambda$
%     $J^k\supseteq J^i$ and
%     $\bfa^k_j>_{\zeta^k_j} \bfa^i_j$
%     for $j\in J^i$ and some $\zeta^k_j\in\lambda$.
% \begin{itemize}
%     \item We can define by induction the canonical limit
%     $\bfa^*$, as follows: 
    
%     $XXX$
%     and $\bfB^{\bfc}\supseteq \bigcup_{i<\lambda, j\in J^i}\bfB^{\bfa^i_j}$.
%     \item  Assume that $(c^i)_{i<\lambda}$ is a continuous 
%     $<_{\IS}$-increasing sequence such that 
%     $c^i$ is $\bar\bfa^i$-good and $c^{i+1}>_{\bar\bfa^i} c^i$.

%     Then any limit $\bfc$ of $(c^i)_{i<\lambda}$ with
%     $\bfB^{\bfc}\supseteq \bfB^{\bfa^*}$
%     satisfies $\bfc>_\AP \bfa^*$.
% \end{itemize}
% \end{lemma}

We now turn to spoiling $(f,A)$:
\begin{definition}
    Given $f\in\Sym(\lambda)$ 
    %such that $f\restriction \delta^b\in \Sym(\delta^b)$ 
    and
    $A\in[\lambda]^\lambda$, 
    we define $c>^{f,A} b$
    by: 
    %$f\restriction\delta^ {b}\in\Sym(\delta^ {b})$,
    $c>_\IS b$, $f\restriction\delta^ {c}\in\Sym(\delta^ {c})$, and there is a $\xi^*\in A\cap \delta^ {c}\setminus \delta^b$ with $f(\xi^*)\ne \pi^{c}(\xi^*)$.

    We write $c>^{f,A}_{\bar\bfa} b$ for: 
    $c>_{\bar\bfa} b\ \&\ c>^{f,A} b$
\end{definition}

% Let us say ``$b\in \IS$ approximates $\bfa\in\AP$'' if $\pi^b=\bfpi^\bfa\restriction \delta^b$
% and $C^b\cup 
% \{\delta^b\}=\bfC^\bfa\cap(\delta^b+1)$. Obviously if 
% $\bfa$ is a limit of $(b_i)_{i<\lambda}$
% than each $b_i$ approximates $\bfa$; if $b,c$ both 
% approximate $\bfa$, then  $b\le_\IS c$ or  $b\ge_\IS c$;
% and if $c\ge_\IS b$ and $c$ approximates $\bfa$,
% then $b$ approximates $\bfa$.
% \begin{lemma}\label{lem:dideldum}
%     Assume $\bfb\in  \AP$
%     and that for unboundedly many 
%     $\delta<\lambda$ the following holds:
%     There is a $b\in\IS$ with 
%     $\delta^b=\delta$, and there is a 
%     $c>^{f,A} b$ 
%      approximating $\bfb$.

%     Then 
%     %If $\bfb$ is a limit of %$(b_i)_{i<\lambda}$
%     %where unboundedly often %$b_{i+1}>^{f,A}b_i$
%     %then 
%     there is an $A'\in  [A]^\lambda$ 
%     such that 
%     $\bfb'$ spoils $(f,A)$, where $\bfb'$ is the same as $\bfb$
%     but with $A'$ added to $\bfB$.
% \end{lemma}
\begin{lemma}\label{lem:dideldum}
    Assume 
    $(b_i)_{i\in\lambda}$ is $<_IS$-increasing such that unboundedly often
    $b_{i+1}>^{f,A}b_i$.
    Then for some
    $A'\in  [A]^\lambda$,
    every limit $\bfb$ of $(b_i)_{i\in\lambda}$ with $A'\in \bfB^{\bfb}$
    spoils $(f,A)$.
\end{lemma}
\begin{proof}
    By taking a subsequence, we can assume that for all odd $i$
    (i.e., $i=\delta+2n+1$ with
    $\delta$ limit or 0 and $n\in\omega$)
    $b_{i+1}>^{f,A} b_{i}$.

    % Assume we already have $b_i$ for $i$ even (i.e., $i=\delta+2n$ with
    % $\delta$ limit or 0 and $n\in\omega$).
    % Then let $(b_{i+1},b_{i+2})$ be a pair with
    % $\delta^{b_{i+1}}>\delta^{b_i}$ 
    % and
    % $b_{i+2}>^{f,A} b_{i+1}$ approximates $\bfb$. Set $I_{i+1}:=\delta^{b_{i+2}}\setminus \delta^{b_{i+1}}$ and let 
    % $\zeta_{i+1}\in I_{i+1}$ be such that
    % $f(\zeta_{i+1})\ne \pi^{b_{i+2}}(\zeta_{i+1})=\bfpi^\bfb(\zeta_{i+1})$.
    % This ends the inductive construction.

    % As all 
    % %Let $J$ be the unbounded set.
    % %subset of $\lambda$ where 
    % $b_{i+1}>^{f,A}b_i$.
    % For $i\in J$, 
    For $i$ odd, set $I_i:=\delta^{b_{i+1}}\setminus \delta^{b_i}$ and let 
    $\xi_i\in I_i$ satisfy 
    $f(\xi_i)\ne \pi^{b_{i+1}}(\xi_i)=\bfpi^{\bfb}(\xi_i)$.

    If $i$ is odd, then 
    $\bfpi^\bfb\restriction I_{i}\in\Sym(I_{i})$ and $f\restriction \delta^{b_{i+1}}\in \Sym(\delta^{b_{i+1}})$.
    
    So if $i<j$ are both odd, then 
    $f(\zeta_{j})>\delta^{b_{i+1}}>\bfpi^\bfb(\zeta_{i})$; and if $j<k$ are both odd
    then $f(\zeta_{j})<\delta^{b_{j}}\le \bfpi^\bfb(\zeta_{k})$.
    This means that $f(\zeta_{j})$ is different
    to all $\bfpi^\bfb(\zeta_{i})$ for $i$ odd.
    
    So we can set 
    $A'=\{\zeta_{j}:\, j\text{ odd}\}$
    and get that $f[A']$ is disjoint to 
    $\bfpi^{\bfb}[A']$. So $\bfb$ with $A'$ added to $\bfB$ spoils $(f,A)$.
\end{proof}

\begin{lemma}\label{lem:purespoil}
    If $\bar\bfa$ is short, $b\in\IS$, 
    $\delta^b$ good for $\bar\bfa$,
    $f\in\Sym(\lambda)$ 
    and
    $A\in[\lambda]^\lambda$, 
    then there is some 
    $d>_{\bar \bfa}^{f,A}b$.
\end{lemma}
\begin{proof}
    Let $\bfB:=\bigcup_{i\in J}\bfB^{\bfa_{i}}$.
    Let $\zeta_0<\lambda$ be the supremum of all $\bfC^{\bfa_i}$-successors
    of $\delta^b$.

    Set $E:=\{\zeta\in\lambda:\, f\restriction\zeta\in\Sym(\zeta)\}$ (a club-set).
    Pick $\zeta_1\in E$ such that $|A\cap (\zeta_1\setminus  \zeta_0) |> |2^\bfB|$.
    Pick  $c>_{\bar\bfa} b$ with $\delta^{c}\in E\setminus\zeta_1$
    and such that $C^c=C^b\cup\{\delta^b\}$.

    Set $I:=\delta^{c}\setminus \zeta_0$.
    For $\alpha,\beta$ in $I\cap A$ set 
    $\alpha\sim  \beta$ iff 
    $(\forall A\in \bfB)\, (\alpha\in A\leftrightarrow \beta\in A)$.
    As there are at most $|2^{\bfB}|$ many equivalence classes, there have to be
    $\beta_0\ne\beta_1$ in $I\cap A$ with $\beta_0\sim\beta_1$. 

    If $\pi^c(\beta_i)\ne f(\beta_i)$ for $i=0$ or $i=1$, set $d:=c$.
    Otherwise, defines $d$ as follows: $\delta^d=\delta^c$, $C^d=C^c$,
    and $\pi^d(\alpha):=\begin{cases}
        \pi^c(\beta_1)&\text{if }\alpha=\beta_0,\\
        \pi^c(\beta_0)&\text{if }\alpha=\beta_1,\\
        \pi^c(\alpha) & \text{otherwise.}
    \end{cases}$
    
    Set $I:=\delta^d\setminus\delta^b$.
    As $\beta_0\sim\beta_1$ we have $\pi^d[A\cap I]=\pi^c[A\cap I]=\bfpi^{\bfa_i}[A\cap I]$ for all $i\in J$ and $A\in\bfB^{\bfa_i}$
    (as $c>_{\bar \bfa} b$).
    
    And as the $\beta_0,\beta_1$ are above $\zeta_0$, and $I_{\delta^b}^{\bfa_i}$ 
    is below $\zeta_0$ for all $i\in J$,
    we have
    $\pi^d\restriction I_{\delta^b}^{\bfa_i}=\pi^c\restriction I_{\delta^b}^{\bfa_i}=\bfpi^{\bfa_i}\restriction I_{\delta^b}^{\bfa_i}$.

    So $d>_{\bar \bfa }b$.
\end{proof}

\section{\texorpdfstring{$2^\lambda=\lambda^+$}{2 to the lambda equals lambda plus} for \texorpdfstring{$\lambda$}{lambda} inaccessible implies a nowhere trivial automorphism}

\begin{lemma}\label{lem:iteration2}
    Every increasing sequence in $\APp$ of length
    ${<}\lambda^+$ has an upper bound.
\end{lemma}

\begin{proof}
  We can assume without loss of generality that the increasing sequence  is
  $\bar a:=(\bfa_i)_{i\in \xi }$ with  $\xi \le\lambda$.

  For $i<\xi $,
  enumerate\footnote{with lots of repetitions} 
  $\bfB^{\bfa_i}$ as $\{x_i^j:\, j\le \lambda\}$,
  and set $B^j_i:=\{x_i^k:\, k\le j\}$
  for $j<\lambda$.
  We enumerate in a way so that
  the $B^j_i$ are increasing with $i<\xi $.
  Let $\bfa_i^j$ be $\bfa_i$ with $\bfB$ replaced by $B^j_i$,
  and for $\ell<\lambda$ set $\bar \bfa^\ell:=(\bfa^\ell_k)_{k<\min(\ell,\xi )}$.
  Note that $\bar \bfa^\ell$ is short.
  
  $\bfc\in\AP$ is an upper bound of $\bar\bfa$
  iff it is an upper bound of all 
  $\bfa^\ell_k$ for $\ell< \lambda$ and $k<\min(\ell,\xi )$.

  %We call $\alpha<\lambda$ (or: a $c\in\IS$)
  %$\ell$-good, if $\alpha$ (or: $\delta^c$)
  %is  $\bar \bfa^\ell$-good.  
  
  % For each $\ell<\lambda$
  % let $D^\ell$ be the $\bar \bfa^\ell$-good ordinals.
  % Then $\ell<m$ implies $D^\ell\subseteq D^m$.
  % Also, $\bigcap_{\ell<\gamma}D^\ell=D^\gamma$ if $\gamma<\lambda$
  % is a limit.

  We now construct by induction on $\ell<\lambda$ a $<_\IS$-increasing 
  continuous sequence 
  $(c^\ell)_{\ell\in\lambda}$, such that $\delta^{c^\ell}$ is $\bar\bfa^\ell$-good:
  
  \begin{itemize}
      \item At limits $\gamma$ we let $c^\gamma$ be the limit of the $(c^k)_{k<\gamma}$, and note that (by induction) its height it is $\bar\bfa^\gamma$-good.
  \item For $j=\ell+1$, let 
  $E$ be the club set of $\bar\bfa^{\ell+1}$-good ordinals, and choose, as in Lemma~\ref{lemma:tr335}(\ref{item:ISextex})
  $c^{\ell+1}>_{\bar \bfa^\ell} c^\ell$ with $\delta^{c^{\ell+1}}\in E$.
  \end{itemize}
  %Note that $c^\ell $ is $\bfa^\ell$-good for all $\ell<\lambda$.  
  Let $\bfc$ be the limit of the $c^\ell$ with $\bfB^{\bfc}:=\bigcup_{i< \xi }\bfB^{\bfa_i}$. 
  
  We claim that $\bfc\ge_\AP \bfa^\ell_j$ for 
  all $\ell< \lambda$ and $j<\min(\ell,\xi)$. 
  Assume that $k>\max(i,j)$.
  \begin{itemize}
      \item By Lemma~\ref{lemma:tr335}(\ref{item:tr331}):
      
      $\delta^{c^k}$ (which is $\bar\bfa^k$-good
      and so, by definition, 
      $\bfa^k_j$-good) is 
      $\bfa^\ell_j$-good, 
      as $\bfa^k_j>_{\delta^{c^k}} \bfa^\ell_j$.
      
      Also, $c^{k+1}>_{\bar \bfa^k} c^k$, so
      (by definition)  
      $c^{k+1}>_{\bfa^k_j} c^k$,
      and 
      so $c^{k+1}>_{\bfa^\ell_j} c^k$.
      \item By Lemma~\ref{lemma:tr335}(\ref{item:oldlem}) we get $\bfc>_\AP \bfa^\ell_j$, as required.\qedhere
  \end{itemize}
\end{proof}

\begin{lemma}\label{lem:spoilsport}
    Given $\bfa\in\APp$, $f\in\Sym(\lambda)$ and $A\in[\lambda]^\lambda$, there is a $\bfb\ge_\AP\bfa$ which is in $\APp$
    and spoils $(f,A)$.
\end{lemma}
    
\begin{proof}
Enumerate $\bfB^\bfa$ as $\{x^j:\, j\in\lambda\}$
and let $\bfa^j$ be $\bfa$ with $\bfB$ replaced
by $\{x^i:\, i<j\}$. So $\bfa^j\in \AP_{<\lambda}$.
    We construct a continuous increasing sequence
    $b^i$ ($i<\lambda$) in $\IS$ such that 
    $\delta^{b^i}$ is $\bfa^i$-good:
    Given $b^{i}$, we find
$b^{i+1}>^{f,A}_{\bfa^i} b^i$ as in Lemma~\ref{lem:purespoil}.
Let $\bfb$ be the limit of the $b^{i}$
with $\bfB^\bfb=\bfB^\bfa\cup\{A'\}$ as in Lemma~\ref{lem:dideldum}.

And $\bfb>_AP \bfa^j$ for all $j<\lambda$ and therefore 
$\bfb>_\AP \bfa$.
\end{proof}

We can now easily show:
%Theorem~\myte, i.e.:
\begin{theorem}\label{thm:a}
\thetext
\end{theorem}

\begin{proof}
    
We construct, by induction on $i\in\lambda^+$, an increasing chain
of $\bfa_i$ in $\APp$, such that:
\begin{itemize}
    \item For limit $i$, we take limits according to Lemma~\ref{lem:iteration2}.
    \item\label{item:zwei} For odd successors $i=j+1=\delta+2n+1$ ($\delta$ limit, $n\in\omega$),
     pick by bookkeeping some $X_j$ and 
     let $\bfa_{j+1}$ be the same as $\bfa_j$ but with 
     $X_j$ and $(\bfpi^{\bfa_j})^{-1}[X_j]$ added to $\bfB$.
    \item For even successors $i=j+1=\delta+2n+2$,
    we pick by 
    book-keeping 
    an $f_j\in\Sym(\lambda)$ and an $A_j\in [\lambda]^\lambda$. Then
    we choose 
    $\bfa_{j+1}\ge_\AP \bfa_j$ 
    spoiling $(f_j,A_j)$, using Lemma~\ref{lem:spoilsport}.
\end{itemize}
Then $\btphi:=\bigcup_{i<\lambda} \btpi^{\bfa_i}$ is a nowhere trivial
automorphism according to Fact~\ref{fact:central}.\end{proof}

\section{Forcing a nowhere trivial automorphism with \texorpdfstring{$2^\lambda>\lambda^+$}{2 to the lambda large}, \texorpdfstring{$\lambda$}{lambda}~inaccessible}

\begin{theorem}\label{thm:b}
    Assume $\lambda$ is inaccessible, $2^\lambda=\lambda^+$ and $\mu>\lambda^+$ 
    is regular.
    Then there is a cofinality preserving (${<}\lambda$-closed
    and $\lambda^+$-cc) poset which forces: $2^\lambda=\mu$, and there is a nowhere trivial
    automorphism of $\mathcal P(\lambda)/[\lambda]^{<\lambda}$.
\end{theorem}

For the rest of this section we fix a $\mu$ as in the lemma.

We will construct a ${<}\lambda$-support iteration 
$(P_\alpha,Q_\alpha)_{\alpha< \mu}$.
We call the final limit $P$. We 
denote the $P_\alpha$-extension $V[G_\alpha]$ by $V_\alpha$.

Each $Q_\alpha$ and therefore also
each $P_\alpha$ will be ${<}\lambda$-closed.

So $x\in \AP$, $x<_\AP y$, as well as $\IS$ (as set) are absolute between $P_\alpha$-extensions (and $|\IS|=\lambda$).

Each $Q_\alpha$ will add a $\bfa^*_\alpha\in \AP$,
such that the $\bfa^*_\alpha$ are $<_\AP$-increasing 
in $\alpha$.

\bigskip

By induction we assume we live in the $P_{\alpha}$-extension $V_\alpha$
where we already have the increasing sequence 
    $(\bfa^*_i)_{i<\alpha}$.
    (We do not claim that this sequence has an upper bound in $V_\alpha$.)

We 
now define $Q_\alpha$, which we will just call $Q$ to improve readability. 
%It will consist of bounded approximations
%to the generic $\bfa_\alpha\in \AP$.
    
\begin{definition}
 $q\in Q$ consists of:
        \begin{enumerate}
        \item 
            A $b^q\in \IS$, also called ``trunk of $q$''.
            
            We also write $\delta^q$, $\pi^q$
            $C^q$ and $I^q_\beta$  instead of
            $\delta^{b^q}$ etc. 
            \item A set $X^q\in[\alpha]^{<\lambda}$, and 
            for $\beta\in X^q$, a set 
            $\bfB^q_\beta\in [\bfB^{\bfa^*_\beta}]^{<\lambda}$,
             such that the $\bfB^q_\beta$ are increasing in
             $\beta$.
             \item For $\beta\in X^q$ 
             set $\bfa^q_\beta$ to be 
             $\bfa^*_\beta$ with $\bfB$ replaced by
             $\bfB^q_\beta$.
             Set $\bar \bfa^q:=(\bfa^q_\beta)_{\beta\in X^q}$
             (which is short).
             \item We require $\delta^{b^q}$ to be good for 
             $\bar \bfa^q$.
         % \item We also require:
         % $f_\alpha\restriction \delta^q\in\Sym(\delta^q)$, and 
         % for every $\beta\in C^q$ there is an
         % $\alpha^*_\beta\in A_\alpha\cap I^q_\beta$
         % with $f(\alpha^*_\beta)\ne \bfpi^q(\alpha^*_\beta)$.
         \end{enumerate}
\end{definition}
(``Short'' and ``good'' are defined in Definition~\ref{def:good}.) 
%We write $\bfa'_\beta$ for the 
%XXXXXXXXXXX
As we use $Q$ as forcing poset, we follow the 
notation that $r\le_Q q$ means that $r$ is stronger than $q$ (whereas in $<_\AP$ and $<_\IS$
the stronger object is the larger one).
\begin{definition}\label{def:uwehtwe}
     $r\le_Q q$ if:
        \begin{enumerate}
    \item $b^r\ge_{\bar\bfa^q} b^q$ (see Definition~\ref{def:strongext}).
            \item $X^r\supseteq X^q$, and
            $\bfB^r_\beta \supseteq \bfB^q_\beta$
            for $\beta\in X^q$.
    %         \item\label{item:oinwge4} Between $\delta^q$ and $\delta^ r$, $r$ honors the promises of $q$.
            
    %         I.e.: If $\gamma\in \bfC^r\setminus \delta^q$ and $\beta\in X^q$
    %         and $A\in \bfB^q_{\beta}$, then
    %         \begin{enumerate}
    %             \item $\gamma\in \bfC^{\bfa_{\beta}}$,
    %             and  
    % $f_\alpha\restriction\gamma\in \Sym(\gamma)$.
    %             \item $\bfpi^r\restriction I^{\bfa_{\beta}}_\gamma=
    %             \bfpi^{\bfa_{\beta}}\restriction I^{\bfa_{\beta}}_\gamma$,
    %             \item $\bfpi^r[A \cap I^r_\gamma]=
    %             \bigcup_{\zeta\in I^r_\gamma\cap Q^{\bfa_{\beta}}}\bfpi^{\bfa_{\beta}}[A \cap I^{\bfa_{\beta}}_{\zeta}]$.
    %         \end{enumerate}    

    % %, and
    % %there is an $\alpha_\gamma\in I^r_\gamma\cap %A_\alpha$
    % %such that $\bfpi^r(\alpha)\ne f(\alpha)$.          
    \end{enumerate}
\end{definition}

The following follows immediately from the definitions:
\begin{fact}
    Assume that $r\le_Q q$, $b\in \IS$  and that $\delta^b$ is good for $\bar\bfa^r$. Then 
    $c\ge_{\bar \bfa^r} b$ implies
    $c\ge_{\bar \bfa^q} b$.
\end{fact}
% \begin{proof}
%     We have to show that $c\ge_{\bfa^q_i} b$
%     for all $i\in X^q$.
%     But $\bfa^q_i\le_{\delta}\bfa^r_i$ for 
%     all $\delta\in C^r$, and we have
%     $c\ge_{\bfa^r_i} b$, which implies
%     $c\ge_{\bfa^q_i} b$.

%     CHECK
% \end{proof}

This implies that $\le_Q$  is transitive. (It
even is a  partial order.)

\begin{lemma}\label{lem:dense} For $q\in Q$, the following holds (in $V_\alpha$):
Let $E\subseteq \lambda$ be club.
    \begin{enumerate}
        % \item\label{item:bla2}  
        % There is an
        % $r<_Q q$ with $\delta^r\in E$.
        % %For $\gamma<\lambda$ there is an
        % %$r<_Q q$ with $\delta^r>\gamma$. 
        \item\label{item:bla3} For $\beta<\alpha$ 
        and $A\in \bfB^{\bfa^*_\beta}$ there is an $r<_Q q$ with $\delta^r\in E$, $\beta\in X^r$ and  $A\in\bfB^r_\beta$.
        \item\label{item:blax} For any $A\in [\lambda]^\lambda$ and $f\in\Sym(\lambda)$
        (both in $V_\alpha$) there is an 
        $r\le_Q q$
        with $b^r<^{f,A} b^q$. 
        %and a
        %$\alpha^*\in A\cap \delta^r\setminus \delta^q$ 
        %such that $\bfpi^r(\alpha^*)\ne f(\alpha^*)$
        %and $f\restriction \delta^r\in\Sym(\delta^r)$.
            
            \item\label{item:Qventer} $Q$ is $\lambda$-centered, witnessed by the function that maps $q$ to its trunk, $b^q$.
            
            (Actually, even ${<}\lambda$ many conditions with the
            same trunk have lower bound.)
            \item\label{item:bla1} $Q_\alpha$ is ${<}\lambda$-closed.
            
            Moreover, a sequence $(q_i)_{i\in\xi}$ ($\xi<\lambda$) has a canonical limit $r$, and the
            trunk of $r$ is the union of the trunks of the $q_i$.
            \end{enumerate}
\end{lemma}

\begin{proof}
    (\ref{item:bla3}): 
    Extend $\bar\bfa^q$ in the obvious
    way to $\bar\bfa^r$:
    Add $\beta$ to the index set, 
    set 
    $\bfB^r_\beta:=\{A\}\cup \bigcup_{\zeta\in X^q\cap (\beta+1)}\bfB^q_\zeta$,
    and add $A$ to all  $\bfB^q_\zeta$
    for $\zeta\in X^q\setminus \beta$.
    %and 
    %for $\zeta\in X^q\setminus \beta$ add $A$ to $\bfB^{\bfa^q_\zeta}$.
    Let $E':=\{\zeta\in\lambda:\, \zeta\text{ good for }\bar\bfa^r\}$.
    Then $E'$ is club according to Fact~\ref{fact:goodisclub}, 
    so we can use
    Lemma~\ref{lemma:tr335}(\ref{item:ISextex})
    to find $b^r>_{\bar\bfa^q}b^q$ with $\delta^r\in E\cap E'$.
    
    (\ref{item:blax}) This is Lemma~\ref{lem:purespoil}.

    (\ref{item:Qventer}) Let $(q_i)_{i\in\mu}$, $\mu<\lambda$
    all have the same trunk $b$. Then the following $r$ is a 
    condition in $Q$:
    $b^r=b$, $X^r=\bigcup_{i<\mu}X^{q_i}$ and 
    $B^r_\zeta= \bigcup_{i<\mu\ \&\ \zeta\in X^{q_i}}B^{q_i}_\zeta$.

    (\ref{item:bla1}) Let $(q_i)_{i<\zeta}$ 
    with $\zeta<\lambda$ be $<_Q$-decreasing.
    Then the obvious union $r$ is an element of $Q$
    and stronger than each $q_i$: 
    
    $b^r$ is the union of the $b^{q_i}$, 
    %i.e.:
    %$\delta^r:=\bigcup_{i<\zeta} \delta^{q_i} $
    %$\bfC^r:=\bigcup_{i<\zeta} \bfC^{q_i} $,
    %and $\bfpi^r:=\bigcup_{i<\zeta} \bfpi^{q_i} $.
    %This is in $\IS$, 
    as in Fact~\ref{fact:istriv1}, and 
    $X^r:=\bigcup_{i<\zeta} X^{q_i} $
    and $\bfB^r_\beta:=\bigcup_{i<\zeta, \beta\in X^ {q_i}} \bfB^{q_i}_\beta$ for each $\beta\in X^r$.

    %XXXXXXXXXX CHECK AND/OR REMOVE XXXXXXXXX:
    Then $\delta^r$ is good for $\bfa^r_\beta$ 
    %Note that $\bfa^{q_i}<_{\delta^r} \bfa^{q_j}$ for $i<j$.
    %So 
    for $\beta\in X^r$:
    %We have to show that $\delta^r$ is good for $\bfa^r_\beta$.
    It is enough to show that $\delta^r$ is good for all
    $\bfa^{q_i}_\beta$ (for sufficiently large $i$). 
    Fix such an $i$. If $j>i$, then
    $\delta^{q_j}$ is good for 
    $\bar \bfa^{q_j}$ and therefore for
    $\bfa^{q_j}_\beta$ and therefore for 
    $\bfa^{q_i}_\beta$. So the limit $\delta^r$ is good as well.

    Similarly one can argue that $b^r>_{\bar \bfa^{q_i}} b^{q_i}$ for all $i<\zeta$.
    %:
    %Fix $i<\zeta$ and $\beta\in X^{1_i}$. We have to show
    %$b^r>_{\bfa^{q_i}_\beta} b^{q_i}$.
    %We show by induction on $j>i$ that
    %$b^{q_j}>_{\bfa^{q_i}_\beta} b^{q_i}$ (then this
    %also holds for the limit $b^r$).
    %$b^{q_{j+1}}>_{\bfb} b^{q_j}$ for $\bfb=\bfa^{q_j}_\beta$
    %and therefore also for $\bfb=\bfa^{q_i}_\beta$,
    %so $b^{q_{j+1}}>_{\bfa^{q_i}_\beta} %b^{q_j}>_{\bfb=\bfa^{q_i}_\beta}\bfa^{q_i}$ by induction.
   \end{proof}

\begin{definition}
    Let $G(\alpha)$ be $Q_\alpha$-generic.
    We define $\bfa^*_\alpha$ (in $V_{\alpha+1}$) as follows:
    
    $\bfC^{\bfa^*_\alpha}:=\bigcup_{q\in G(\alpha)} C^q$,
    $\bfpi^{\bfa^*_\alpha}:=\bigcup_{q\in G(\alpha)} \pi^q$,
    and $\bfB^{\bfa^*_\alpha}:=P(\lambda)$.
\end{definition}

%It is now easy to see that this is also a limit
%of the $\bfa_\beta$ for $\beta<\alpha$:

\begin{lemma} $P_{\alpha+1}$ forces:
\begin{enumerate}
    \item\label{item:uqwhrq} $\bfa^*_\alpha>_\AP \bfa^*_\beta$ for all $\beta<\alpha$.
    \item \label{item:uqwhrq2} $\bfa^*_\alpha$ spoils $(f,A)$ for all
    $(f,A)\in V_\alpha$.
\end{enumerate}
\end{lemma}
The proof consists of straightforward density 
arguments:
\begin{proof}
    For~(\ref{item:uqwhrq}) we 
    know that by
    there is some $q\in G(\alpha)$ with
    $\beta\in X^q$. This implies that
    $\bfC^{\bfa^*_\alpha}\subseteq \bfC^{\bfa^*_\beta}$ 
    above $\delta^q$ and that
    $\bfpi^{\bfa^*_\alpha}\restriction  I^{\bfa^*_\beta}_\zeta=
    \bfpi^{\bfa^*_\beta}\restriction  I^{\bfa^*_\beta}_\zeta$
    for all $\zeta\in \bfC^{\bfa^*_\beta}\setminus \delta^q$.
    We can also assume that a given $A\in\bfB^{\bfa^*_\beta}$
    is in $\bfB^q_\beta$, which implies that 
    $\bfpi^{\bfa^*_\alpha}[A]=\bfpi^{\bfa^*_\beta}[A]$
    above $\delta^q$.

    For~(\ref{item:uqwhrq2}) and $(f,A)\in V_\alpha$ we know by
    Lemma~\ref{lem:dense}(\ref{item:blax})
    that for $q\in G(\alpha)$ of unbounded  heights 
    there are 
    $r(q)$ in $G(\alpha)$ such that
    $b^{r(q)} >^{f,A} b^q$.
    I.e, in $V_{\alpha+1}$, $\bfa_\alpha^*$ is a limit 
    of an $<_\IS$-increasing sequence as in Lemma~\ref{lem:dideldum}, 
    therefore $\bfa^*_\alpha$ spoils $(f,A)$
    (as $A'$ certainly is in $\bfB^{\bfa^*_\alpha}=P(\lambda)$).
\end{proof}    

So $P$ adds a sequence $(\bfa^*_\alpha)_{\alpha<\mu}$
that we can use in Fact~\ref{fact:central} to get 
a nowhere trivial automorphism. 
We will now show that $P$ is $\lambda^+$-cc, which finishes
the proof of Theorem~\ref{thm:b}.

\begin{lemma}
    Set $t(p):=(b^{p(\alpha)})_{\alpha\in\dom(p)}$ (i.e., the sequence
    of trunks). Then the following set $D$ is dense: $p$ 
    in $D$ if there is an $x\in V$ such that the empty
    condition forces $t(p)=x$.
\end{lemma}
    
\begin{proof}
    %$t(p):=(b^{p(\alpha)})_{\alpha\in\dom(p)}$ is a ground
    %model element. 
    %Moreover, the following set $D$ is dense: $p$ 
    %in $D$ if there is an $x\in V$ such that the empty
    %condition forces $t(p)=x$.
    
    We claim that the lemma holds for $P_\alpha$,
    by induction on $P_\alpha$. Successors and 
    limits of cofinality ${\ge}\lambda$ are clear.
    
    Let $\alpha$ be a limit with cofinality $\kappa< \lambda$,
    and $(\alpha_i)_{i\in\kappa}$ cofinal in $\alpha$,
    $\alpha_0=0$. Set $D_j:=D\cap P_{\alpha_j}$
    (by induction dense in $P_{\alpha_j}$).
    We construct by induction on $j\in\kappa$
    a decreasing sequence $p_j\in P_\alpha$ such that
    $p_0=p$ and
    $p_j\restriction \alpha_j\in D$:
    
    Successors:
    Given $p_j$, we find $r\le p_j\restriction \alpha_{j+1}$
    in $D_{j+1}$ and set $p_{j+1}:=r\wedge p_j$
    (which is the same as $r\wedge p$).

    Limits:
    Given $(p_i)_{i<\xi}$ with $\xi\le\kappa$, let
    $p_\xi$ be the pointwise canonical limit. 
    Note that we can calculate (in $V$) each $p_{\xi}(\beta)$
    from the sequence $(p_i(\beta))_{i<\xi}$
    (it is just the union).
    % $P_\alpha$ is ${<}\lambda$-closed, and elements 
    % of $\IS$ are ${<}\lambda$-sequences.
    % So the following set $D\subseteq P$ is dense:
    % $p\in D$ if
    % the sequence $(b^{p(\alpha)})_{\alpha\in\dom(p)}$
    % is in $V$.
    % Assume $p,q$ in $D$ and $\dom(p)\cap\dom(q)=:X$.
    % If $(b^{p(\alpha)})_{\alpha\in X}$ and
    % $(b^{q(\alpha)})_{\alpha\in X}$ are the same,
    % then $p$ and $q$ are compatible. FILL
\end{proof}

\begin{lemma}(Assuming $2^\lambda=\lambda^+$ in the ground model.)
    $P$ is $\lambda^+$-cc.
\end{lemma}

\begin{proof}
    Assume $(a_i)_{i\in\lambda^+}$ is a sequence in $P$.
    For every $a_i$ find an $a'_i\le a$ in $D$.
    By Fodor (or the Delta-system lemma)
    there is an $X\subseteq \lambda^+$
    of size $\lambda^+$ such that 
    $\{\dom(a'_i):\, i\in X\}$ form a Delta system
    with heart $\Delta$,
    and furthermore we can assume that $t(a'_i)\restriction \Delta$ 
    (the sequence 
    of trunks restricted to $\Delta$)
    is the same for all
    $i\in X$.
    (There are $\lambda^{|\Delta|}=\lambda<\lambda^+$
    many such restrictions.)
    Then for $i,j$ in $X$, the conditions 
    $a'_i$ and $a'_j$ (and therefore also $a_i$
    and $a_j$) 
    are compatible.
    % It is enough to show that there is an $X\subseteq \lambda^+$
    % of size $\lambda^+$ such that $i,j$ in $X$ implies 
    % that $a'_i$ and $a'_j$ are compatible,
    % By thinning out we can assume that $\{\dom(p):\, p\in A\}$ form a $\Delta$-system with a heart we call $\Delta$,
    % i.e., $\Delta\in[\alpha^*]^{<\lambda}$ for some $\alpha^*<\mu$
    % and $\dom(p)\cap \dom(q)=\Delta$ for $p\ne q$ in $A$.
    % By thinning out again we can assume that the sequence
    % $(b^{p(\alpha)})_{\alpha\in\Delta}$ is constant for all
    % $p\in A$. So each $p,q$ in $A$ is compatible.
\end{proof}

\begin{remark}
    Generally, preserving $\lambda^+$-cc for $\lambda>\omega_1$ is much more cumbersome than for $\lambda=\omega$, as there is no obvious universal theorem analogous to
    ``the finite support iteration of ccc forcings is ccc''. In our case, it was very easy to show $\lambda^+$-cc manually.
    However, we could have used existing iteration theorems. We give two examples (but there surely are many more). Note that the following theorems do not require $\lambda$
    to be inaccessible.
\begin{enumerate}[leftmargin=*]\item
    From~\cite{MR1716019} 
    (generalising the $\lambda=\aleph_1$ case from~\cite[Lem.\ 4.1]{MR0823775}):     \begin{itemize}
        \item Definition~\cite[p.\ 237]{MR1716019}: $Q$ is $\lambda$-centered closed, if a centered
    subset $D$ of $Q$ of size ${<}\lambda$ has a lower bound.
    \item Lemma~\cite[p.\ 237]{MR1716019}:
    Assume $2^{<\lambda}=\lambda$. Let 
    $P$ %$(P_{\alpha},Q_{\alpha})_{\alpha<\beta}$ 
    be a ${<}\lambda$-support iteration such
    that each iterand
%$Q_{\alpha}$ 
is (forced to be) $\lambda$-linked and $\lambda$-centered closed. Then 
$P$
%$P_{\beta}$ 
is
$\lambda^{+}$-cc.
    \end{itemize}
    It is easy to see that our $Q$ satisfies the
    requirements ($Q$ is even $\lambda$-centered and ``$\lambda$-linked closed'').

    \item 
    From \cite{Sh:1144} (generalizing the $\lambda=\aleph_1$ case from~\cite[3.1]{Sh:80}):
    \begin{itemize}
        \item \cite[Def.\  2.2.2]{Sh:1144}: $Q$ is ``stationary $\lambda^+$-Knaster'',
         if for every sequence $(p_i)_{i < \lambda^+}$
         in $Q$ there exists a club
$E \subseteq \lambda^+$ and a regressive function
$f$ on $E \cap S^{\lambda^+}_\lambda$ such that $p_i$ and $p_j$ are compatible 
whenever 
$f(i) = f(j)$.
        \item \cite[Lem.\ 2.2.5]{Sh:1144}: Assume that $P$ is a ${<}\lambda$-support iteration of iterands that all are: 
        stationary $\lambda^+$-Knaster,   strategically ${<}\lambda$-closed, and any two compatible conditions have a greatest lower bound, as do decreasing $\omega$-sequences.
        Then $P$ is stationary $\lambda^+$-Knaster.
    \end{itemize}
    Note that our $Q$ satisfies
    the requirements, and that
    our proof of $\lambda^+$-cc
    actually shows stationary $\lambda^+$-Knaster.
\end{enumerate}
\end{remark}

\bibliographystyle{amsalpha}
\bibliography{shlhetal,1251}
\end{document}